\documentclass[11pt]{amsart}

\usepackage{graphicx}
\usepackage{epsfig}
\usepackage{amsmath}
\usepackage{autobreak}
\usepackage{amssymb}
\usepackage{bigints}
\usepackage[]{breqn}
\usepackage{tikz}
\usepackage{graphicx}
\usepackage{epsfig}
\usepackage[noadjust]{cite}
\usepackage{xcolor}
\usepackage{tikz}
\usetikzlibrary{matrix}
\usepackage[colorlinks,citecolor=red,urlcolor=blue,bookmarks=false,hypertexnames=true]{hyperref}

\theoremstyle{plain}
\newtheorem{theorem}      {Theorem}      [section]
\newtheorem*{theorem*}    {Theorem \eqref{thm:appl}}
\newtheorem{proposition}  [theorem]  {Proposition}
\newtheorem{corollary}    [theorem]  {Corollary}

\theoremstyle{definition}

\newtheorem{remark}       [theorem]  {Remark}

\def \i{\mbox{${\textnormal{i}}$}}
\def \d{\mbox{${\textnormal{d}}$}}
\def \dz{\mbox{${\stackrel{o}{\textnormal{d}}}$}}
\def \Dz{\mbox{${\stackrel{o}{\Delta}}$}}

\def \r{\mbox{${\mathbb R}$}}
\def \s{\mbox{${\mathbb S}$}}

\numberwithin{equation}{section}

\begin{document}

\title[The energy density of biharmonic quadratic maps between spheres]
{The energy density of biharmonic quadratic maps between spheres}

\author{Rare\c s Ambrosie and Cezar Oniciuc}

\address{Faculty of Mathematics\\ Al. I. Cuza University of Iasi\\
Blvd. Carol I, 11 \\ 700506 Iasi, Romania} \email{rares_ambrosie@yahoo.com}

\address{Faculty of Mathematics\\ Al. I. Cuza University of Iasi\\
Blvd. Carol I, 11 \\ 700506 Iasi, Romania} \email{oniciucc@uaic.ro}



\begin{abstract}
In this paper, we first prove that a quadratic form from $\s^m$ to $\s^n$ is non-harmonic biharmonic if and only if it has constant energy density $(m+1)/2$. Then, we give a positive answer to an open problem raised in \cite{AOO23} concerning the structure of non-harmonic biharmonic quadratic forms. As a direct application, using classification results for harmonic quadratic forms, we infer classification results for non-harmonic biharmonic quadratic forms.
\end{abstract}

\maketitle

\section{Introduction}

Biharmonic maps between Riemannian manifolds represent a natural generalization of the well known harmonic maps. Even if the most examples and classification results for biharmonic maps have been obtained in the submanifolds theory (see, for example, \cite{FO22} and \cite{OC20}), there are also other geometric contexts where the biharmonic maps have interesting applications (see, for example, \cite{BK03}, \cite{BFO10}, \cite{BFO17}, \cite{BO18}, \cite{MOR15}, \cite{O03}, \cite{O22} and \cite{WOY14}). Recently, biharmonic homogeneous polynomial maps between spheres have been studied and, in particular, the quadratic forms proved to be a good environment for the biharmonic equation (see \cite{AOO23}).

It is well known that there are quadratic forms with non-constant energy density, but a quadratic form is harmonic if and only if its energy density is a certain constant (see Proposition \eqref{prop3}). Naturally, one could ask if the same type of property holds for biharmonic quadratic forms. In this paper we give a positive answer to this question and prove that a quadratic form between unit Euclidean spheres is non-harmonic biharmonic if and only if its energy density is equal to $(m+1)/2$, where $m$ is the dimension of the domain sphere (see Theorem \eqref{TH3}). Moreover, we prove that any non-harmonic biharmonic quadratic form is obtained from a quadratic form which lies, as a harmonic map, in a certain small hypersphere of the target sphere, confirming the supposition stated as an Open Problem in \cite{AOO23} (see Theorem \eqref{TH4}). This rigidity result allows us to derive classification results for non-harmonic biharmonic quadratic forms from known results for harmonic quadratic forms (see Theorems \eqref{TH5} and \eqref{TH6}).

\section{Preliminary}

As suggested by J. Eells and J.H. Sampson in  ~\cite{ES64,ES65}, or J. Eells and L. Lemaire in \cite{EL83}, biharmonic maps $\varphi:\left(M^m, g\right)\to \left(N^n, h\right)$ between two Riemannian manifolds are critical points of the bienergy functional
$$
E_2:C^\infty(M,N)\to \mathbb{R}, \qquad E_2(\varphi) = \frac{1}{2}\int_{M}\left|\tau(\varphi)\right|^2 \ v_g,
$$
where $M$ is compact and $\tau(\varphi) = \textnormal{trace}\nabla d\varphi$ is the tension field associated to the smooth map $\varphi$. In 1986, G.Y. Jiang proved in  \cite{J86,J86-2} that the biharmonic maps are characterized by the vanishing of their bitension field, that is
$$
0 = \tau_2(\varphi) = -\Delta\tau(\varphi) - \textnormal{trace} R^N\left(d\varphi(\cdot),\tau(\varphi)\right)\d\varphi(\cdot).
$$
The equation $\tau_2(\varphi) = 0$ is called the biharmonic equation and it is a forth order semilinear elliptic equation.

In this paper, the following sign conventions for the rough Laplacian, that acts on the set $C\left(\varphi^{-1}TN\right)$ of all sections of the pull-back bundle $\varphi^{-1}TN$, and for the curvature tensor field are used
$$
\Delta \sigma = -\textnormal{trace}\nabla^2\sigma, \quad R(X,Y)Z = \nabla_X\nabla_Y Z - \nabla_Y\nabla_YZ -\nabla_{[X,Y]}Z.
$$
By $\s^m(r)$ we indicate the $m$-dimensional Euclidean sphere of radius $r$ and, when $r = 1$, we write $\s^m$ instead of $\s^m(1)$.

Since any harmonic map is automatically biharmonic, we study the biharmonic maps which are not harmonic. Such maps are called proper biharmonic.

Many examples of proper biharmonic maps were obtained when the target manifold is the Euclidean sphere and the maps take values in a small hypersphere of it. Taking into account the expression of the second fundamental form of a small hypersphere of the target sphere, we have

\begin{theorem}\label{TH1}
Let $\psi:M\to\s^{n-1}(r)$ be a non constant map, and consider $\varphi = \i\circ\psi: M\to\s^{n}(R)$, where $0<r<R$ and $\i:\s^{n-1}(r)\to\s^n(R)$ is the inclusion map. Then, the bitension field of $\varphi$ is given by
\begin{align}\label{ecuatia1}
\tau_2(\varphi) =& \tau_2(\psi) + 2\left(\frac{1}{R^2}-\frac{1}{r^2}\right)\d\psi\left(\textnormal{grad}\left|\d\psi\right|^2\right) + 4\left(\frac{1}{R^2} -\frac{1}{r^2}\right)e(\psi)\tau(\psi) \\
                 & + \frac{\sqrt{R^2-r^2}}{Rr}\left(2\Delta\left(e(\psi)\right) - 2\textnormal{div}\theta^\sharp + \left|\tau(\psi)\right|^2 - 4\left(\frac{2}{R^2} - \frac{1}{r^2}\right)\left(e(\psi)\right)^2\right)\eta,\nonumber
\end{align}
where $\theta(X) := \langle\d\psi(X), \tau(\psi)\rangle$.
\end{theorem}
We have the following consequences
\begin{corollary}\label{coro1}
Assume that $\varphi$ is biharmonic.
\begin{enumerate}
  \item If the map $\psi$ is harmonic with constant energy density, then $r = R/\sqrt{2}$.
  \item If $M$ is compact, then $r\geq R/\sqrt{2}$. Moreover, if $r = R/\sqrt{2}$, then $\psi$ is harmonic and has constant energy density (see \cite{LO07}).
  \item If $e(\varphi)$ is constant and $\textnormal{div}\theta^\sharp = 0$, then $r \geq R/\sqrt{2}$. Moreover, if $r=R/\sqrt{2}$, then $\psi$ is harmonic.
\end{enumerate}
\end{corollary}

In this paper, we consider homogeneous polynomial maps of degree $2$ between spheres, that are quadratic forms (we follow the notations and terminology in \cite{AOO23}).

Consider the diagram below
\bigskip
\begin{center}
  \begin{tikzpicture}
  \matrix (m) [matrix of math nodes,row sep=3em,column sep=4em,minimum width=2em]
  {
     \r^{m+1} & \r^{n+1} \\
     \s^m & \s^n \\};
  \path[-stealth]
    (m-2-1) edge node [left] {$\i$} (m-1-1)
    (m-1-1) edge node [above] {$F$} (m-1-2)
    (m-2-1) edge node [above] {$\Phi$} (m-1-2)
    (m-2-2) edge node [right] {$\i$} (m-1-2)
    (m-2-1) edge node [below] {$\varphi$} (m-2-2);
\end{tikzpicture}
\end{center}
where $F:\r ^{m+1}\rightarrow\r ^{n+1}$ is a quadratic form. Then, $F$ can be written in the form
$$
F(\overline{x}) = \left(X^t A_1 X, X^t A_2 X, \ldots ,X^t A_{n+1}\ X\right),
$$
where the vector $\overline x = \left(x^1, x^2, \ldots , x^{m+1}\right)$ is identified with the matrix
$$
X^t = \left[x^1 \ x^2 \ \ldots \ x^{m+1}\right]
$$
and $A_1$, $\ldots$, $A_{n+1}$ are square matrices of order $m+1$ such that if $|\overline x| = 1$, then $|F(\overline x)| = 1$. In this case, we have $|F(\overline x)|^2 = |\overline x|^4$. We note that, by a standard symmetrization process, we can always assume that $A_i$ is symmetric, $i = 1, 2, \ldots, n+1$.

We will always assume that $\varphi$ is not a constant map, therefore there exists $i_0\in\{1,2,\ldots,n+1\}$ such that $A_{i_0}$ is not the identity matrix $I_{m+1}$ multiplied by a non-zero real constant.

With $\dz$ and $\Dz$ we denote operators that act on $\r^{m+1}$. Since $F$ is a homogeneous polynomial map of degree $2$, we have at any $\overline x \in \r^{m+1}$
\begin{align}\label{ecuatia2}
\dz F(\overline x) &= \begin{bmatrix}
                        2 X^t A_1 \\
                        2 X^t A_2 \\
                        \vdots \\
                        2 X^t A_{n+1}
                       \end{bmatrix}, \nonumber\\
\left|\dz F(\overline x)\right|^2 &= 4 X^t\left(A_1^2 + A_2^2 + \dots + A_{n+1}^2\right) X = 4 X^tSX,\\
\Dz F &= -2\left(\textnormal{tr}A_1, \textnormal{tr}A_2, \ldots , \textnormal{tr}A_{n+1}\right), \nonumber
\end{align}
where we denoted $S = A_1^2 + A_2^2 + \dots + A_{n+1}^2$.

The quadratic maps have the following general property.

\begin{proposition}\label{prop1}
Let $F:\r ^{m+1}\rightarrow\r ^{n+1}$ be a quadratic form. Then,
\begin{align*}
8\textnormal{tr}S + \left|\Dz F\right|^2 = 4(m + 1)(m+3).
\end{align*}
\end{proposition}

Further, we recall a known result (see, for example, \cite{BW03} and \cite{ER93}).

\begin{proposition}\label{prop2}
The map $\varphi$ is harmonic if and only if $\Dz F = \overline 0.$
\end{proposition}

The next result relates the harmonicity of the map $\varphi$ with its energy density $e(\varphi) = |\d\varphi|^2/2$ and, for the sake of completeness, we give its proof.

\begin{proposition}\label{prop3}
The map $\varphi$ is harmonic if and only if its energy density $e(\varphi) = m + 1$.
\end{proposition}

\begin{proof}
First, we recall the following formulas for quadratic forms, which hold on $\s^m$:
\begin{align}\label{ecuatia3}
\tau(\varphi) &= -\Dz F + \left(\left|\dz F\right|^2 - 2(m+3)\right)\Phi
\end{align}
and
\begin{align}\label{ecuatia4}
\left|\d\varphi\right|^2 &= \left|\dz F\right|^2 - 4.
\end{align}
For the direct implication, it is easy to see that it follows directly from Proposition \eqref{prop2} and formulas \eqref{ecuatia3} and \eqref{ecuatia4}.

For the converse implication, Equations \eqref{ecuatia3} and \eqref{ecuatia4} give us
$$
\tau(\varphi) = -\Dz F,
$$
which further implies that
$$
\left\langle\Dz F, \Phi\right\rangle = 0.
$$
Thus, on $\s^m$ we have
$$
\textnormal{tr}A_1\cdot X^tA_1X + \textnormal{tr}A_2\cdot X^tA_2X + \dots + \textnormal{tr}A_{n+1}\cdot X^tA_{n+1}X = 0,
$$
which is equivalent to
$$
X^t\left(\textnormal{tr}A_1\cdot A_1 + \textnormal{tr}A_2\cdot A_2 + \dots + \textnormal{tr}A_{n+1}\cdot A_{n+1}\right)X = 0.
$$
Since the above equation is a homogeneous one, we can extend it on $\r^{m+1}$ and, because the matrix $A = \textnormal{tr}A_1\cdot A_1 + \textnormal{tr}A_2\cdot A_2 + \dots + \textnormal{tr}A_{n+1}\cdot A_{n+1}$ is a symmetric matrix, we can conclude that $A = O_{m+1}$. We notice that
$$
\textnormal{tr}A = \frac{1}{4} \left|\Dz F\right|^2 = 0.
$$
Thus, $\tau(\varphi) = 0$ and the proof is complete.
\end{proof}

\begin{remark}
The condition $e(\varphi) = m + 1$ is equivalent to $S = \left((m+3)/2\right)I_{m+1}$.
\end{remark}

Next, we recall the following result.
\begin{theorem}\label{TH2}{(see \cite{AOO23})}
Let $F:\r ^{m+1}\rightarrow\r ^{n+1}$ be a quadratic form given by
$$
F(\overline{x}) = \left(X^t A_1 X, X^tA_2X, \ldots ,X^t A_{n+1}\ X\right),
$$
such that if $|\overline x| = 1$, then $\left|F(\overline x)\right| = 1$.  We consider $\varphi:\s^m\rightarrow\s^n$ defined by $\varphi(\overline x) = F(\overline x)$ and $\Phi = \i\circ\varphi:\s^m\rightarrow\r^{n+1}$. If we denote $S = A_1^2 + A_2^2 + \dots + A_{n+1}^2$, then at a point $\overline x\in\s^m$, the bitension field of $\varphi$ has the following expression
\begin{align}\label{ecuatia5}
\tau_2(\varphi)_{\overline x}  =& -4\left(m + 5 - 4X^tSX\right)\left(\textnormal{tr}A_1, \textnormal{tr}A_2, \ldots , \textnormal{tr}A_{n+1}\right) \nonumber\\
               & + 4\left((m+3)(m+5) - 6(m + 5)X^tSX + 8 \left(X^tSX\right)^2 \right) \Phi(\overline x) \\
               & + 32 \left(X^t A_1S X, X^tA_2SX, \ldots ,X^t A_{n+1}S X\right).\nonumber
\end{align}
\end{theorem}

We note that, the condition $S = \alpha I_{m+1}$, for some real constant $\alpha > 1$, is equivalent to $|\d\varphi|^2$ is constant. More precisely, $|\d\varphi|^2 = 4(\alpha - 1)$. In this case the bitension field takes the form
$$
\tau_2(\varphi)_{\overline x}  = -8\left(\alpha - \frac{m + 5}{4}\right)\Dz F + 32\left(\alpha - \frac{m+5}{4}\right)\left(\alpha - \frac{m+3}{2}\right) \Phi(\overline x).
$$

\begin{corollary}\label{coro2}
If the quadratic form $\varphi$ has constant energy density, then $\varphi$ is proper biharmonic if and only if $e(\varphi) = (m+1)/2$.
\end{corollary}

The following relations will be useful in our proof. According to \cite{GT95}, by using the standard coordinates, any quadratic form $F:\r^{m+1}\to\r^{n+1}$ can be written as
$$
F(\overline x) = \sum_{i=1}^{m+1}\overline a_i\left(x^i\right)^2 + \sum_{1\leq i<j\leq m+1}\overline a_{ij}x^ix^j,
$$
where $\overline a_i\in\r^{n+1}$, for $i=1,\ldots,m+1$, and $\overline a_{ij} \in \r^{n+1}$, for $1\leq i<j\leq m+1$. To simplify the notation, we set $\overline a_{ij} = \overline a_{ji}$, in order to have defined the vectors $\overline a_{ij}$ for all $1\leq i,j\leq m+1$. The condition that $F$ takes $\s^m$ into $\s^n$ is equivalent to the following (see \cite{GT95} and \cite{W68}):
\begin{align}\label{ecuatia6}
  |\overline a_i| & = 1,\quad\textnormal{ for } 1\leq i \leq m+1,\nonumber\\
  \langle \overline a_i, \overline a_{ij}\rangle &= 0,\quad\textnormal{ for } i\neq j,\nonumber\\
  |\overline a_{ij}|^2 + 2\langle\overline a_i, \overline a_j\rangle &= 2,\quad\textnormal{ for } i\neq j,\\
  \langle \overline a_i, \overline a_{jk}\rangle + \langle \overline a_{ij}, \overline a_{ik}\rangle &= 0,\quad\textnormal{ for } i,j,k \textnormal{ distinct},\nonumber\\
  \langle \overline a_{ij}, \overline a_{kl}\rangle + \langle \overline a_{ik}, \overline a_{jl}\rangle + \langle \overline a_{il}, \overline a_{jk}\rangle &= 0,\quad \textnormal{ for } i,j,k,l \textnormal{ distinct.}\nonumber
\end{align}

\section{Energy density of biharmonic quadratic forms}

In this section we first prove that a biharmonic quadratic map has constant energy density. We note that in the proof, the computations for the coefficients of certain polynomials have been done using \textit{Mathematica}\textsuperscript{\textregistered}.

\begin{theorem}\label{TH3}
Let $\varphi:\s^m\rightarrow\s^n$ be a quadratic form. Then, $\varphi$ is proper-biharmonic if and only if its energy density $e(\varphi)$ is constant and equal to $(m+1)/2$.
\end{theorem}
\begin{proof}
For the direct implication, the idea is to use Formula \eqref{ecuatia5} and the homogenization process.

The energy density and the notions of harmonic and biharmonic maps are invariant under the action of isometries on domain and/or codomain. Therefore, we are allowed to perform linear orthogonal transformation on $\r^{m+1}$ and/or $\r^{n+1}$. The strategy of the proof is to compute $\tau_2(\varphi)$ and then to act with an orthogonal transformation on $\r^{m+1}$ in order to get a simplified form for $\tau_2(\varphi)$. Further, we transform the non-homogeneous polynomial equation $\tau_2(\varphi) = \overline 0$ on $\s^m$ into a homogeneous polynomial equation of degree $6$ on $\r^{m+1}$ valued in $\r^{n+1}$, and thus, for each component of this equation, all coefficients have to be $0$.

We suppose from the beginning that $\varphi$ is not harmonic, thus
$$
\Dz F \neq \overline 0 \quad \textnormal{and} \quad S \neq \frac{m+3}{2}I_{m+1}.
$$
The matrix $S$ defines a quadratic map. We perform an orthogonal change of the domain variables $x^1$, $x^2$, $\ldots$, $x^{m+1}$  which brings $S$ in diagonal form. Thus, we write $S$ as
$$
S  = \begin{bmatrix}
             s_1 & 0 & \dots & 0 \\
             0 & s_2 & \dots & 0 \\
             \vdots & \vdots & \ddots & \vdots \\
             0 & 0 & \dots & s_{m+1}
     \end{bmatrix}.
$$
We write the matrices $A_i$ as
$$
A_i  = \frac{1}{2}\begin{bmatrix}
             2a^i_1 & a^i_{12} & \dots & a^i_{1(m+1)} \\
             a^i_{12} & 2a^i_2 & \dots & a^i_{2(m+1)} \\
             \vdots & \vdots & \ddots & \vdots \\
             a^i_{1(m+1)} & a^i_{2(m+1)} & \dots & 2a^i_{m+1}
           \end{bmatrix},\quad \textnormal{for } 1\leq i \leq n+1.
$$
The vectors $\overline a_i$ and $\overline a_{ij}$ satisfy System \eqref{ecuatia6}. By direct computation of $S$ and using the first relation of System \eqref{ecuatia6}, we get
\begin{align}\label{ecuatia7}
s_k &= \left|\overline a_k\right|^2 + \frac{1}{4}\sum_{i=1,\ i\neq k}^{m+1}\left|\overline a_{ik}\right|^2\nonumber\\
    &= 1 + \frac{1}{4}\sum_{i=1,\ i\neq k}^{m+1}\left|\overline a_{ik}\right|^2.
\end{align}
We transform the non-homogeneous polynomial map $\tau_2(\varphi)$ from Equation \eqref{ecuatia5} into a homogeneous polynomial map of degree $6$ because it is well known that if a homogeneous polynomial vanishes on the sphere $\s^m$, then it vanishes on $\r^{m+1}$. Thus, we obtain
\begin{align}\label{ecuatia8}
& -4\left|\overline x|^4((m + 5)|\overline x|^2 - 4X^tSX\right)\left(\textnormal{tr}A_1, \textnormal{tr}A_2, \ldots , \textnormal{tr}A_{n+1}\right) \nonumber\\
& + 4\left((m+3)(m+5)|\overline x|^4 - 6(m + 5)|\overline x|^2X^tSX + 8 \left(X^tSX\right)^2 \right) F(\overline x) \\
& + 32|\overline x|^4 \left(X^t A_1S X, X^tA_2SX, \ldots ,X^t A_{n+1}S X\right) = \overline 0, \quad \textnormal{on } \r^{m+1}\nonumber
\end{align}
We analyse the coefficient list for each component of the above homogeneous polynomial equation. For any $i\in\{1,2,\ldots,n+1\}$ we notice that the coefficient of $(x^k)^6$, which has to vanish, gives
\begin{align}\label{ecuatia9}
4\left(5 + m - 4s_k\right)\left(a^i_k\left(3 + m - 2s_k\right) - \textnormal{tr}A_i\right) = 0,\quad \forall k\in\{1,2,\ldots,m+1\},
\end{align}
Thus, for any $k$ arbitrarily fixed, we have either
$$
s_k = \frac{m+5}{4},
$$
or
$$
a^i_k\left(3 + m - 2s_k\right) - \textnormal{tr}A_i = 0,\quad \forall i\in\{1,2,\ldots,n+1\}.
$$
The last relations can be rewritten as
\begin{align}\label{ecuatia10}
\overline a_k\left(3 + m - 2s_k\right) &=-\frac{1}{2}\Dz F,
\end{align}
and further, taking in consideration the first relation in Equation \eqref{ecuatia6}, we obtain
\begin{align}\label{ecuatia11}
\left|3 + m - 2s_k\right| &=\frac{1}{2}\left|\Dz F\right|.
\end{align}

Without loss of generality, we assume that $s_1 = (m+5)/4$. By looking at the $i$-th component of Equation \eqref{ecuatia8}, from the coefficients of $\left(x^1\right)^5x^2$, $\left(x^1\right)^5x^3$, $\ldots$, $\left(x^1\right)^5x^{m+1}$ we get the following system of equations, for any $i\in\{1,2,\ldots,n+1\}$,
\begin{align}\label{ecuatia12}
\begin{cases}
-4a^i_{12}\left(5 + m - 4s_2\right) = 0,\\
-4a^i_{13}\left(5 + m - 4s_3\right) = 0,\\
\dots \\
-4a^i_{1(m+1)}\left(5 + m - 4s_{m+1}\right) = 0.
\end{cases}
\end{align}
Since $s_1$ is greater than $1$, equation \eqref{ecuatia7} gives us
$$
\left|\overline a_{12}\right| + \left|\overline a_{13}\right| + \ldots + \left|\overline a_{1(m+1)}\right| > 0.
$$
Without loss of generality, we assume that $a^{i_0}_{12} \neq 0$. Thus, from Equation \eqref{ecuatia12} for $i = i_0$ we obtain $s_2 = (m + 5)/4$. By looking at the $i_0$-th component of Equation \eqref{ecuatia8}, from the coefficients of $\left(x^1\right)^3x^2\left(x^3\right)^2$, $\left(x^1\right)^3x^2\left(x^4\right)^2$, $\ldots$, $\left(x^1\right)^3x^2\left(x^{m+1}\right)^2$ we get
\begin{align*}
\begin{cases}
2a^{i_0}_{12}(5 + m)\left(5 + m - 4s_3\right) = 0,\\
2a^{i_0}_{12}(5 + m)\left(5 + m - 4s_4\right) = 0,\\
\dots \\
2a^{i_0}_{12}(5 + m)\left(5 + m - 4s_{m+1}\right) = 0.
\end{cases}
\end{align*}
Then,
$$
s_3 = s_4 = \dots = s_{m+1} = \frac{m+5}{4}.
$$
Thus $S = \left((m+5)/4\right)I_{m+1}$.

We conclude that if there is an eigenvalue of $S$ equal to $(m+5)/4$, then all of them are equal to $(m+5)/4$. Thus, $S = \left((m+5)/4\right)I_{m+1}$ and using Equations \eqref{ecuatia2} and \eqref{ecuatia4} we get that $\varphi$ has constant energy density $e(\varphi) = (m + 1)/2$.

Now, we assume  that $s_k \neq (m+5)/4$, for any $k\in\{1,2,\ldots, m+1\}$. We have that $\Dz F \neq \overline 0$, otherwise, according to Proposition \eqref{prop2} $\varphi$ would be harmonic. From Equations \eqref{ecuatia10} it follows that $3 + m - 2s_k \neq 0$ for any $k\in\{1,2,\ldots, m+1\}$ and therefore
\begin{equation}\label{ecuatia13}
\overline a_k = \frac{-1}{2\left(3 + m - 2s_k\right)}\Dz F, \quad k\in\{1,2,\ldots,m+1\}.
\end{equation}
From Equation \eqref{ecuatia13} it follows that $\overline a_k \parallel \overline a_l$, for any $k,l\in\{1,2,\ldots,m+1\}$. Since $\left|\overline a_k\right| = 1$, for any $k\in\{1,2,\ldots,m+1\}$, we distinguish two cases:

\emph{Case 1.} For any $k,l\in\{1,2,\ldots,m+1\}$, $\overline a_k = \overline a_l$. From  Equations \eqref{ecuatia10} we conclude that
$$
s_1 = s_2 = \dots = s_{m+1} = \alpha,
$$
and thus $S$ is a scalar matrix, i.e. $S = \alpha I_{m+1}$, which implies that the map $\varphi$ has constant energy density. Since $\alpha \neq (m+5)/4$, using Corollary \eqref{coro2}  we conclude the map $\varphi$ cannot be proper biharmonic.

\emph{Case 2.} There exists $k\neq l$ such that $\overline a_k \neq \overline a_l$. Without loss of generality, we assume that $\overline a_1 = -\overline a_2$. From Equation \eqref{ecuatia13} we get
\begin{align}\label{ecuatia14}
s_2 = 3 + m - s_1
\end{align}
Also, from the third relation of System \eqref{ecuatia6} we get
$$
\left|\overline a_{12}\right|^2 = 2 - 2\left\langle\overline a_1, \overline a_2\right\rangle = 4.
$$
Thus, there exists $i_0$ such that  $a^{i_0}_{12} \neq 0$. From the coefficients of $\left(x^1\right)^5x^2$  in the $i_0$-th component of Equation \eqref{ecuatia8} we get
\begin{align}\label{ecuatia15}
4a^{i_0}_{12}\left(15 + m^2 + m\left(8 - 6s_1\right) - 26s_1 + 8s^2_1 + 4s_2\right) = 0.
\end{align}
Using Equation \eqref{ecuatia14} we substitute $s_2$ in Equation \eqref{ecuatia15} and we obtain the equation
$$
8s_1^2 - 6s_1(5 + m) + m^2 + 12m +27 = 0,
$$
 which has the solutions
$$
\left\{\frac{m + 3}{2}, \frac{m + 9}{4}\right\}.
$$
Since $3 + m - 2s_k \neq 0$ for any $k\in\{1,2,\ldots, m+1\}$, it follows that $s_1 = (m+9)/4$, the coefficient of $\left(x^1\right)^3\left(x^2\right)^3$ becomes $-4a^{i_0}_{12}(-3+m)$ which vanishes if and only if $m = 3$. But, then $s_1 = 3$ which is in contradiction with $3 + m - 2s_1 \neq 0$.
Thus, \textit{Case 2} leads only to contradictions.

The converse implication follows immediately from Corollary \eqref{coro2} and the proof is complete.
\end{proof}

From Theorem \eqref{TH3}, Equation \eqref{ecuatia4} and Proposition \eqref{prop1} we deduce

\begin{corollary}
Let $\varphi:\s^m\rightarrow\s^n$ be a proper biharmonic quadratic form. Then, $\left|\Dz F\right|^2 = 2(m + 1)^2.$
\end{corollary}

\begin{remark}
We note that the direct implication of Proposition \eqref{prop3} can be proved in the same manner as the direct implication of Theorem \eqref{TH3} using Equation \eqref{ecuatia3} and the homogenization process.
\end{remark}

\begin{remark}
There exists quadratic forms with constant energy density which are neither biharmonic nor harmonic. We give the following example from \cite{AOO23}. Let $F_\lambda:\r^4\times\r^4\rightarrow\r^6$ given by
$$
F_\lambda(\overline z,\overline w) = \left(|\overline z|^2 + \lambda |\overline w|^2, \sqrt{2(1-\lambda)}P(\overline z,\overline w), \sqrt{1 - \lambda^2}|\overline w|^2\right),
$$
where $\lambda \in [0,1)$ and $P:\r^4\times\r^4\rightarrow\r^4$ is the orthogonal multiplication of quaternions.
By direct computations, we can see that
$$
S = A_1^2 + A_2^2 + \dots + A_8^2 = (3 - 2\lambda) I_8.
$$
Thus, the quadratic form $\varphi_\lambda:\s^7\rightarrow\s^5$, obtained from the restriction of $F_\lambda$, has constant energy density $e\left(\varphi_\lambda\right) = 4(1-\lambda)$. If $\lambda = 0$, the map $\varphi_0$ is proper biharmonic, but if $\lambda\in (0,1)$, the map $\varphi_\lambda$ is neither biharmonic nor harmonic. Moreover, $\varphi_0\left(\s^7\right)$ lies in the small hypersphere $\s^{4}\left(1/\sqrt{2}\right)$.
\end{remark}

All known examples of proper biharmonic quadratic forms suggested the following (see \cite{AOO23}).

\vspace{0.3cm}

\textbf{\emph{Open Problem.}} If $\varphi:\s^m\rightarrow\s^n$ is a proper biharmonic quadratic form, then up to an isometry of $\s^n$, the first $n$ components of $\varphi$ are harmonic polynomials on $\r^{m+1}$ and form a map $\psi:\s^m\rightarrow\s^{n-1}(1/\sqrt{2})$.

\vspace{0.3cm}

Now, using the results presented above, we can give a positive answer to this problem.

\begin{theorem}\label{TH4}
If $\varphi:\s^m\rightarrow\s^n$ is a proper biharmonic quadratic form, then up to an isometry of $\s^n$, the first $n$ components of $\varphi$ form a harmonic map $\psi:\s^m\rightarrow\s^{n-1}(1/\sqrt{2})$.
\end{theorem}

\begin{proof}
First, we recall that $\varphi$ is proper biharmonic if and only if its energy density $e(\varphi)$ is constant and equal to $(m+1)/2$. We recall the following formula,
$$
\tau(\varphi) = -\Dz F + \left(\left|\dz F\right|^2 - 2(m+3)\right)\Phi.
$$
Denoting $\overline a = \Dz F$ and $\alpha = \left|\dz F\right|^2 - 2(m+3)$, from the above equation we obtain
\begin{equation*}
\left\langle\overline a, \Phi\right\rangle = \alpha.
\end{equation*}
We note that $\alpha\neq0$, otherwise $\varphi$ would be harmonic.
Thus, the image $\varphi(\s^m)$ lies in a small hypersphere of $\s^n$ of radius $r$, where
$$
r^2 = 1 - \frac{\alpha^2}{\left|\overline a\right|^2}.
$$

We perform a rotation of $\r^{n+1}$ on the components of $F$ such that
$$
\Dz F = -2\left(\textnormal{tr}A_1, \ldots, \textnormal{tr}A_n, \textnormal{tr}A_{n+1}\right) = (0,\ldots,0, a), \quad a\neq 0,
$$
and the last component of $\Phi$ is now equal to $\sqrt{1-r^2}$. The first $n$ components of $\Phi$ are harmonic polynomials on $\r^{m+1}$ and form a harmonic map $\psi:\s^m\rightarrow\s^{n-1}(r)$.

Since $\psi:\s^m\to\s^{n-1}(r)$ is harmonic, it follows that $e(\psi) = r^2(m+1)$. On the other hand, $e(\varphi) = e(\psi) = (m+1)/2$ and therefore $r=1/\sqrt{2}$.

Or, we can conclude using Corollary \eqref{coro1}.
\end{proof}

In the case of quadratic forms, Corollary \eqref{coro1} has a more rigid form.

\begin{proposition}
Let $\varphi:\s^m\to\s^n$ be a proper biharmonic quadratic map. Assume that $\varphi$ lies in a small hypersphere $\s^{n-1}(r)$ of $\s^n$. Then, $r\geq 1/\sqrt{2}$. Moreover,
\begin{enumerate}
  \item if $r = 1/\sqrt{2}$, then the map $\psi: \s^m\to\s^{n-1}\left(1/\sqrt{2}\right)$ determined by $\varphi$ is harmonic;
  \item if $r>1/\sqrt{2}$, then $\varphi$ lies in $\s^{n-2}\left(1/\sqrt{2}\right)\subset\s^{n-1}(r)$ and $\psi:\s^m\to\s^{n-2}\left(1/\sqrt{2}\right)$ is harmonic.
\end{enumerate}
\end{proposition}
\begin{proof}
The case $r=1/\sqrt{2}$ was proved in Corollary \eqref{coro1}.

We assume that $r>1/\sqrt{2}$ and $\s^{n-1}(r)\equiv\s^{n-1}(r)\times\left\{\sqrt{1-r^2}\right\}$. In this case, the last component of $\varphi$ is given by the matrix $A_{n+1} = \sqrt{1-r^2}I_{m+1}$ and the image of $\varphi$ is included in the hyperplanes
$$
\left(\Pi\right): y^{n+1} = \sqrt{1-r^2},
$$
and
$$
\left(\Pi_1\right): \left\langle\Dz F,\overline y\right\rangle = -m-1.
$$
Therefore, the image of $\varphi$ lies also in a hypersphere $\s^{n-1}\left(1/\sqrt{2}\right)$ of $\left(\Pi_1\right)$, centered in $A$. In order to find the coordinates of the point $A$, we consider the line passing through the origin and perpendicular to $\left(\Pi_1\right)$, that is
$$
(d): \overline y = t\Dz F.
$$
Therefore, the intersection point $A$ between $(d)$ and $\left(\Pi_1\right)$ is given by
$$
t_0 = -\frac{1}{2(m+1)},
$$
and its position vector is
$$
t_0\Dz F = \left(\frac{\textnormal{tr}A_1}{m+1}, \ldots, \frac{\textnormal{tr}A_n}{m+1}, \sqrt{1-r^2}\right).
$$
Since the last component of $A$ is $\sqrt{1-r^2}$, it follows that $A\in\left(\Pi\right)$. Therefore, $\left(\Pi\right)\cap\left(\Pi_1\right)$ passes through the center of the hypersphere $\s^{n-1}\left(1/\sqrt{2}\right)$ and $\varphi$ lies in a totally geodesic hypersphere of $\s^{n-1}\left(1/\sqrt{2}\right)$, and $\psi:\s^m\to\s^{n-2}\left(1/\sqrt{2}\right)$ is a harmonic map. Clearly, $\s^{n-2}\left(1/\sqrt{2}\right)$ is a small hypersphere of $\s^{n-1}(r)$, and $\varphi$ is not full.
\end{proof}

Next, we present two applications of Theorem \eqref{TH3}, obtaining classification results.

Using the result of Calabi concerning the uniqueness of compact minimal $2$-dimensional round spheres in $\s^n$, i.e. the uniqueness of the Boruvka spheres (see \cite{C67} and also \cite{B85}, \cite{K83}),  we obtain

\begin{theorem}\label{TH5}
Let $\varphi:\s^2\to\s^n$ be a full quadratic map. Assume that $\varphi$ is homothetic. Then, $\varphi$ is proper biharmonic if and only if $n=5$, $\varphi\left(\s^2\right)\subset\s^4\left(1/\sqrt{2}\right)$, and up to homothetic changes of domain and codomain, $\psi:\s^2\to\s^4\left(1/\sqrt{2}\right)$ is the Veronese map.
\end{theorem}

In \cite{G87}, G. Toth obtained the classification of all full harmonic maps from $\s^3$ to $\s^n$.

\begin{theorem}
Full quadratic harmonic maps of $\s^3$ into $\s^n$ exist only if $2\leq n \leq 8$ and $n\neq 3$. Moreover, if $\varphi:\s^3\to\s^n$ is such a map, then there exist $U\in O(4)$, $V\in O(n+1)$ and a symmetric positive definite matrix $B\in\s^2\left(\r^{n+1}\right)$ such that
$$
V\circ\varphi\circ U = B\circ\varphi_n,
$$
where $\varphi_n:\s^3\to\s^n$ is defined by
\begin{equation*}
\varphi_n\left(\overline x\right) =
\begin{cases}
\left(\left(x^1\right)^2 + \left(x^2\right)^2 - \left(x^3\right)^2 - \left(x^4\right)^2, 2\left(x^1x^3 - x^2x^4\right), 2\left(x^1x^4 + x^2x^3\right)\right), & n = 2\\
\left(\left(x^1\right)^2 + \left(x^2\right)^2 - \left(x^3\right)^2 - \left(x^4\right)^2, 2x^1x^3, 2x^1x^4, 2x^2x^3, 2x^2x^4\right), & n = 4\\
\left(\left(x^1\right)^2 - \left(x^2\right)^2, \left(x^3\right)^2 - \left(x^4\right)^2, 2x^1x^2, \sqrt{2}\left(x^1x^3 + x^2x^4\right),\right. & \\
\left. \quad \quad \sqrt{2}\left(x^2x^3 - x^1x^4\right), 2x^3x^4\right), & n = 5\\
\left(\frac{1}{\sqrt{2}}\left(\left(x^1\right)^2 + \left(x^2\right)^2 - \left(x^3\right)^2 - \left(x^4\right)^2\right), \frac{1}{\sqrt{2}}\left(\left(x^1\right)^2 - \left(x^2\right)^2\right), \right. & \\
\left. \quad \quad  \frac{1}{\sqrt{2}}\left(\left(x^3\right)^2 - \left(x^4\right)^2\right), \sqrt{2}x^1x^2, \sqrt{3}\left(x^1x^3 + x^2x^4\right),\right. & \\
\left. \quad \quad \quad \quad \sqrt{3}\left(x^2x^3 - x^1x^4\right), \sqrt{2}x^3x^4\right), & n = 6\\
\left(\left(x^1\right)^2 - \left(x^2\right)^2, \left(x^3\right)^2 - \left(x^4\right)^2, 2x^1x^2, \sqrt{2}x^1x^3, \sqrt{2}x^1x^4,\right. & \\
\left. \quad \quad  \sqrt{2}x^2x^3, \sqrt{2}x^2x^4, 2x^3x^4\right), & n =7\\
\varphi_{\lambda_2}\left(x^1, x^2, x^3, x^4\right), \textnormal{(} \varphi_{\lambda_2} = \textnormal{ a standard minimal immersion)} & n= 8
\end{cases}
\end{equation*}
\end{theorem}

Thus, we can state

\begin{theorem}\label{TH6}
Full quadratic proper biharmonic maps of $\s^3$ into $\s^n$ exist only if $3\leq n \leq 9$ and $n\neq 4$. Moreover, if $\varphi:\s^3\to\s^n$ is such a map, then there exist $U\in O(4)$, $V\in O(n+1)$ and a symmetric positive definite matrix $B\in\s^2\left(\r^{n+1}\right)$ such that
$$
V\circ\varphi\circ U =\left(\frac{1}{\sqrt{2}}B\circ\varphi_n,\frac{1}{\sqrt{2}}\right).
$$
\end{theorem}

\textbf{Acknowledgments.} Thanks are due to Ye-Lin Ou for usefull comments and suggestions.


\begin{thebibliography}{99}
\bibitem{AOO23} R.~Ambrosie, C.~Oniciuc, Y.-L.~Ou, \textit{Biharmonic homogeneous polynomial maps between spheres}, Results Math., to appear.
\bibitem{BFO10} P.~Baird, A.~Fardoun, S.~Ouakkas, \textit{Liouville-type theorems for biharmonic maps between Riemannian manifolds}, Adv. Calc. Var. 3 (2010), no. 1, 49--68.
\bibitem{BFO17} P.~Baird, A.~Fardoun, S.~Ouakkas, \textit{Biharmonic maps from biconformal transformations with respect to isoparametric functions}, Differential Geom. Appl. 50 (2017), 155--166.
\bibitem{BK03} P.~Baird, D.~Kamissoko, \textit{On constructing biharmonic maps and metrics}, Ann. Global Anal. Geom. 23 (2003), no. 1, 65--75.
\bibitem{BO18} P.~Baird, Y.-L.~Ou, \textit{Biharmonic conformal maps in dimension four and equations of Yamabe-type}, J. Geom. Anal. 28 (2018), no. 4, 3892--3905.
\bibitem{BW03} P.~Baird, J.C.~Wood, \textit{Harmonic Morphisms Between Riemannian Manifolds}, London Mathematical Society Monographs (N.S.), 29, Oxford University Press, Oxford, (2003).
\bibitem{B85} R.L.~Bryant, \textit{Minimal surfaces of constant curvature in $S^n$}, Trans. Amer. Math. Soc. 290 (1985), no. 1, 259--271.
\bibitem{C67} E.~Calabi, \textit{Minimal immersions of surfaces in Euclidean spheres}, J. Differential Geometry 1 (1967), 111--125.
\bibitem{EL83} J.~Eells, L.~Lemaire, \textit{Selected Topics in Harmonic Maps}, CBMS Regional Conference Series in Mathematics, American Mathematical Society, vol.50 (1983), Providence, RI,v+85 pp.
\bibitem{ER93} J.~Eells, A.~Ratto, \textit{Harmonic maps and minimal immersions with symmetries. Methods of ordinary differential equations applied to elliptic variational problems}, Princeton University Press, Princeton, NJ, 1993. iv+228 pp.
\bibitem{ES64} J.~Eells, J.H.~Sampson, \textit{Energie et deformations en geometrie differentielle}, Ann. Inst. Fourier 14 (1964), fasc. 1, 61--69.
\bibitem{ES65} J.~Eells, J.H.~Sampson, \textit{Variational Theory in Fibre Bundles}, Proc. U.S.-Japan Seminar in Differential Geometry (1965), pp.22--33.
\bibitem{FO22} D.~Fetcu, C.~Oniciuc, \textit{Biharmonic and biconservative hypersurfaces in space forms}, Contemp. Math., vol. 777 (2022), 65--90.
\bibitem{J86} G.Y.~Jiang, \textit{2-harmonic isometric immersions between Riemannian manifolds}, Chinese Ann. Math. Ser. A 7 (1986), 130--144.
\bibitem{J86-2} G.Y.~Jiang, \textit{2-harmonic maps and their first and second variational formulas}, Chinese Ann. Math. Ser. A7(4) (1986), 389--402.
\bibitem{K83} K.~Kenmotsu, \textit{Minimal surfaces with constant curvature in 4-dimensional space forms}, Proc. Amer. Math. Soc. 89 (1983), no. 1, 133--138.
\bibitem{LO07} E.~Loubeau, C.~Oniciuc, \textit{On the biharmonic and harmonic indices of the Hopf map}, Trans. Amer. Math. Soc. 359 (2007), 5239–-5256.
\bibitem{MOR15} S.~Montaldo, C.~Oniciuc, A.~Ratto, \textit{Rotationally symmetric biharmonic maps between models}, J. Math. Anal. Appl. 431 (2015), no. 1, 494--508.
\bibitem{O02} C.~Oniciuc, \textit{Biharmonic maps between Riemannian manifolds}, An. Ştiinţ. Univ. Al. I. Cuza Iaşi. Mat. (N.S.) 48 (2002), no. 2, 237--248.
\bibitem{O03} C.~Oniciuc, \textit{New examples of biharmonic maps in spheres}, Colloq. Math. 97 (2003), no. 1, 131--139.
\bibitem{O22} Y.-L.~Ou, \textit{Bi-eigenmaps and biharmonic submanifolds in a sphere}, J. Geom. Phys. 180 (2022), Paper No. 104621, 5.
\bibitem{OC20} Y.-L.~Ou, B.-Y.~Chen, \textit{Biharmonic Submanifolds And Biharmonic Maps In Riemannian Geometry}, World Scientific, Hackensack, N. J., (2020).
\bibitem{GT95} G.~Toth, \textit{Quadratic Eigenmaps between Spheres}, Geometriae Dedicata, 56 (1995), 35--52.
\bibitem{G87} G.~Toth, \textit{Classification of quadratic harmonic maps of $\s^3$ into spheres}, Indiana Univ. Math. J. 36 (1987), no. 2, 231--239.
\bibitem{WOY14}  Z.-P.~Wang, Y.-L.~Ou, H.-C.~Yang, \textit{Biharmonic maps from a 2-sphere}, J. Geom. Phys. 77 (2014), 86--96.
\bibitem{W68} R.~Wood, \textit{Polynomial Maps from Spheres to Spheres}, Invent. Math. 5 (1968), 163--168.

\end{thebibliography}
\end{document}